\documentclass{amsart}
\usepackage[demo]{graphicx}
\usepackage{caption}
\usepackage{subcaption}
\usepackage{graphicx, amsmath, amsfonts, epsfig, latexsym, amssymb, amsthm}
\usepackage{tikz}
\usepackage{tkz-berge}
\usepackage{float}

\usetikzlibrary {positioning}
\definecolor {processblue}{cmyk}{0.96,0,0,0}
\newtheorem{thm}{Theorem}[section]
\theoremstyle{definition}
\newtheorem{cor}[thm]{Corollary}
\newtheorem{prop}[thm]{Proposition}
\newtheorem{defn}[thm]{Definition}

\newtheorem{rem}[thm]{Remark}
\newtheorem{ex}[thm]{Example}

\numberwithin{equation}{section}
\begin{document}
\title[Second ideal intersection graph of a commutative ring]
{Second ideal intersection graph of a commutative ring}
\author{ F. Farshadifar}
\address{Department of Mathematics Education, Farhangian University, P.O. Box 14665-889, Tehran, Iran.}
\email{f.farshadifar@cfu.ac.ir}


\begin{abstract}
Let $R$ be a commutative ring with
identity.
In this paper, we introduce and investigate the second ideal intersection graph $SII(R)$ of $R$ with vertices
are non-zero proper ideals of $R$ and two distinct vertices $I$ and $J$ are adjacent if and only if $I \cap J$ is a second ideal of $R$.
\end{abstract}
\maketitle
\section{Introduction}
\noindent
Throughout this paper, $R$ will denote a commutative ring with identity and $\Bbb Z$ will denote the ring of of integers.
Also, "$\subset$" will denote the strict inclusion.

An ideal $I$ of $R$ is said to be a \emph{second ideal} if $I \neq 0$ and for every element $r$ of $R$ we have either $rI= 0$ or $rI= I$ \cite{Y01}.

A \emph{graph} $G$ is defined as the pair $(V(G),E(G))$, where $V(G)$ is the set of vertices of $G$ and $E(G)$ is the set of edges of $G$. For two distinct vertices $a$ and $b$ of $V(G)$, the notation $a-b$ means that $a$ and $b$ are adjacent. A graph $G$ is said to be \emph{complete} if $a-b$ for all distinct $a, b\in V(G)$, and $G$ is said to be \emph{empty} if $E(G) =\emptyset$. Note by this definition that a graph may be empty even if $V (G)\not =\emptyset$. An empty graph could also be described as totally disconnected. If $|V (G)|\geq 2$, a \emph{path} from $a$ to $b$ is a series
of adjacent vertices $a-v_1-v_2-...-v_n-b$. The \emph{length of a path} is the number of edges it contains. A \emph{cycle} is a path that begins and ends at the same vertex in which no edge is repeated, and all vertices other than the starting and ending vertex are distinct. If a graph $G$ has a cycle, the \emph{girth} of $G$ (notated $g(G)$) is defined as the length of the shortest cycle of $G$; otherwise, $g(G) =\infty$. A graph $G$ is \emph{connected} if for every pair of distinct vertices $a, b\in V (G)$, there exists a path from $a$ to $b$. If there is a path from $a$ to $b$ with $a, b \in V (G)$,
then the \emph{distance from} $a$ to $b$ is the length of the shortest path from $a$ to $b$ and is denoted by $d(a, b)$. If there is not a path between $a$ and $b$, $d(a, b) = \infty$. The \emph{diameter} of $G$ is diam$(G) = Sup\{d(a,b)| a, b \in V(G)\}$.
 A graph without any cycles is an \textit{acyclic graph}. A vertex that is adjacent to
every other vertices is said to be a\textit{universal vertex} whereas a vertex with degree
zero is called an \textit{isolated vertex}.

In \cite{saha2023prime}, the authors introduced and investigate the definition of \textit{the prime ideal sum graph} of $R$, denoted
by $PIS(R)$, which is a graph whose vertices are non-zero proper ideals of $R$ and two distinct
vertices $I$ and $J$ are adjacent if and only if $I + J$ is a prime ideal of $R$.
In this paper, we introduce and study the second ideal intersection graph $SII(R)$ of $R$ with vertices
are non-zero proper ideals of $R$ and two distinct vertices $I$ and $J$ are adjacent if and only if $I \cap J$ is a second ideal of $R$.
This can be regarded as a dual notion of the prime ideal sum graph introduced in \cite{saha2023prime}.
Some of the results in this article are dual of the results for the prime ideal sum graph introduced in \cite{saha2023prime}.
\section{Main Results}
\noindent
\begin{defn}\label{2.1}
The \textit{second ideal intersection graph} of $R$, denoted
by $SII(R)$, is an undirected simple graph whose vertices are non-zero proper ideals of $R$ and two distinct
vertices $I$ and $J$ are adjacent if and only if $I \cap J$ is a second ideal of $R$.
This can be regarded as a dual notion of the prime ideal sum graph introduced in \cite{saha2023prime}.
\end{defn}

Let $n$ be a positive integer. Consider the ring $\Bbb Z_n$ of integers modulo $n$. We know that $\Bbb Z_n$ is a principal ideal ring and each of these ideals is generated by $\bar{m} \in \Bbb Z_n$, where $m$ is a factor of $n$. In this paper, we denote this ideal by $\langle m\rangle$.

\begin{ex}\label{000}
Let $R=\Bbb Z_{p^3q}$, where $p$, $q$ are primes. Then the non-zero proper ideals of $R$ are $\langle p \rangle$, $\langle q  \rangle$, $\langle p^2 \rangle$, $\langle pq  \rangle$, $\langle p^3\rangle$, and $\langle p^2q \rangle$. In the following figures,  we can see the graphs $PIS(\Bbb Z_{p^3q})$ and $SII(\Bbb Z_{p^3q})$.
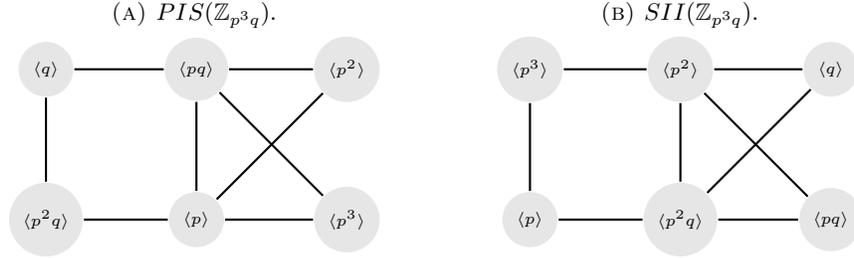
\begin{figure}[H]
\centering
\begin{subfigure}[b]{0.49\textwidth}
\centering
\caption{$PIS(\Bbb Z_{p^3q})$.}
\begin{center}
\begin{tikzpicture}[auto,node distance=2 cm,
  thick,main node/.style={circle,fill=black!10,font=\sffamily\tiny\bfseries}]
\node[main node] (1) {$\langle q \rangle$};
\node[main node] (2) [right of=1] {$\langle pq \rangle$};
\node[main node] (3) [right of=2] {$\langle p^2 \rangle$};
\node[main node] (4) [below of=1] {$\langle p^2q \rangle$};
\node[main node] (5) [right of=4] {$\langle p \rangle$};
\node[main node] (6) [right of=5] {$\langle p^3 \rangle$};
\path[every node/.style={font=\sffamily\small}]
    (1) edge node [left] {} (2)
    (2) edge node [left] {} (3)
    (1) edge node [left] {} (4)
    (4) edge node [left] {} (5)
    (5) edge node [left] {} (6)
    (5) edge node [left] {} (3)
     (5) edge node [left] {} (2)
    (2) edge node [left] {} (6);
        \end{tikzpicture}
\end{center}
\end{subfigure}
\hfill
\begin{subfigure}[b]{0.49\textwidth}
\centering
\caption{$SII(\Bbb Z_{p^3q})$.}
\begin{center}
\begin{tikzpicture}
[auto,node distance=2 cm,
  thick,main node/.style={circle,fill=black!10,font=\sffamily\tiny\bfseries}]
\node[main node] (1) {$\langle p^3 \rangle$};
\node[main node] (2) [right of=1] {$\langle p^2 \rangle$};
\node[main node] (3) [right of=2] {$\langle q \rangle$};
\node[main node] (4) [below of=1] {$\langle p \rangle$};
\node[main node] (5) [right of=4] {$\langle p^2q \rangle$};
\node[main node] (6) [right of=5] {$\langle pq \rangle$};
\path[every node/.style={font=\sffamily\small}]
    (1) edge node [left] {} (2)
    (2) edge node [left] {} (3)
    (1) edge node [left] {} (4)
    (4) edge node [left] {} (5)
    (5) edge node [left] {} (6)
    (5) edge node [left] {} (3)
     (5) edge node [left] {} (2)
    (2) edge node [left] {} (6);
\end{tikzpicture}
\end{center}
\end{subfigure}
\end{figure}
 \end{ex}

 \begin{ex}\label{010}
Let $R=\Bbb Z_{p^2q^2}$, where $p$, $q$ are primes. Then the non-zero proper ideals of $R$ are $\langle p^2 \rangle$, $\langle q p^2 \rangle$, $\langle p \rangle$, $\langle pq  \rangle$, $\langle q\rangle$,  $\langle pq^2\rangle$, and $\langle q^2 \rangle$. In the following figures,  we can see the graphs $PIS(\Bbb Z_{p^2q^2})$ and $SII(\Bbb Z_{p^2q^2})$.
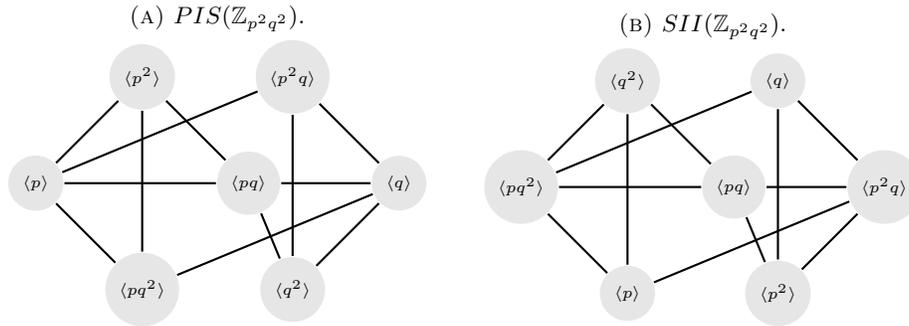
\begin{figure}[H]
\centering
\begin{subfigure}[b]{0.49\textwidth}
\centering
\caption{$PIS(\Bbb Z_{p^2q^2})$.}
\begin{center}
\begin{tikzpicture}[auto,node distance=2 cm,
  thick,main node/.style={circle,fill=black!10,font=\sffamily\tiny\bfseries}]
\node[main node] (1) {$\langle p^2 \rangle$};
\node[main node] (2) [right of=1] {$\langle p^2q \rangle$};
\node[main node] (3) [below left of=1] {$\langle p \rangle$};
\node[main node] (4) [below right of=1] {$\langle pq \rangle$};
\node[main node] (5) [below right of=2] {$\langle q \rangle$};
\node[main node] (6) [below right of=3] {$\langle pq^2 \rangle$};
\node[main node] (7) [below left of=5] {$\langle q^2 \rangle$};
\path[every node/.style={font=\sffamily\small}]
    (1) edge node [left] {} (3)
    (1) edge node [left] {} (6)
    (1) edge node [left] {} (4)
    (2) edge node [left] {} (5)
    (2) edge node [left] {} (7)
    (2) edge node [left] {} (3)
    (3) edge node [left] {} (6)
    (3) edge node [left] {} (4)
    (4) edge node [left] {} (5)
    (6) edge node [left] {} (5)
    (7) edge node [left] {} (5)
    (7) edge node [left] {} (4);
        \end{tikzpicture}
\end{center}
\end{subfigure}
\hfill
\begin{subfigure}[b]{0.49\textwidth}
\centering
\caption{$SII(\Bbb Z_{p^2q^2})$.}
\begin{center}
\begin{tikzpicture}[auto,node distance=2 cm,
  thick,main node/.style={circle,fill=black!10,font=\sffamily\tiny\bfseries}]
\node[main node] (1) {$\langle q^2 \rangle$};
\node[main node] (2) [right of=1] {$\langle q \rangle$};
\node[main node] (3) [below left of=1] {$\langle pq^2 \rangle$};
\node[main node] (4) [below right of=1] {$\langle pq \rangle$};
\node[main node] (5) [below right of=2] {$\langle p^2q \rangle$};
\node[main node] (6) [below right of=3] {$\langle p \rangle$};
\node[main node] (7) [below left of=5] {$\langle p^2 \rangle$};
\path[every node/.style={font=\sffamily\small}]
    (1) edge node [left] {} (3)
    (1) edge node [left] {} (6)
    (1) edge node [left] {} (4)
    (2) edge node [left] {} (5)
    (2) edge node [left] {} (7)
    (2) edge node [left] {} (3)
    (3) edge node [left] {} (6)
    (3) edge node [left] {} (4)
    (4) edge node [left] {} (5)
    (6) edge node [left] {} (5)
    (7) edge node [left] {} (5)
    (7) edge node [left] {} (4);
        \end{tikzpicture}
\end{center}
\end{subfigure}
\end{figure}
 \end{ex}

\begin{ex}\label{0190}
Let $R=\Bbb Z_{p^k}$, where $p$, $q$ are primes and $k$ is a positive integer. Then
$\langle p^{k-1} \rangle$ is the only second ideal of $R$. One can see that $SII(\Bbb Z_{p^k})$
is a star graph with center vertex $\langle p^{k-1} \rangle$.
\end{ex}

\begin{prop}\label{p0190}
Let $R=\Bbb Z_n$, where $n$ is a positive integer. Then we have the following.
\begin{itemize}
\item [(a)] If $n=p_1p_2...p_k$ ($k\geq 3$), where $p_i$'s are distinct prime numbers, then $SII(\Bbb Z_n)$ contains a cycle of length 3.
\item [(b)] If $n=p_1^3p_2...p_k$ ($k\geq 2$), where $p_i$'s are distinct prime numbers, then $SII(\Bbb Z_n)$ contains a cycle of length 3.
\end{itemize}
\end{prop}
\begin{proof}
(a) Clearly, $\langle p_1p_2...p_{k-1} \rangle$ is a second ideal of $R$. Set $I=\langle p_1p_2...p_{k-1} \rangle$, $J=\langle p_1p_2...p_{k-2} \rangle$, and $K=\langle p_{k-1} \rangle$. Then the graph $SII(\Bbb Z_n)$ contains the following cycle.
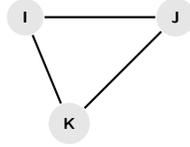
\begin{figure}[H]
\centering
\begin{tikzpicture}[auto,node distance=2 cm,
  thick,main node/.style={circle,fill=black!10,font=\sffamily\tiny\bfseries}]
\node[main node] (1) {I};
\node[main node] (2) [right of=1] {J};
\node[main node] (3) [below left of=2] {K};
\path[every node/.style={font=\sffamily\small}]
    (1) edge node [left] {} (2)
    (2) edge node [left] {} (3)
    (3) edge node [left] {} (1);
\end{tikzpicture}
\end{figure}

(b) Clearly, $\langle p_1^2p_2...p_k \rangle$ is a second ideal of $R$. Set $I=\langle p_1^2p_2...p_k \rangle$, $J=\langle p_2...p_k \rangle$, and $K=\langle p^2 \rangle$. Then the graph $SII(\Bbb Z_n)$ contains the following cycle.
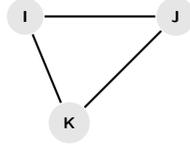
\begin{figure}[H]
\centering
\begin{tikzpicture}[auto,node distance=2 cm,
  thick,main node/.style={circle,fill=black!10,font=\sffamily\tiny\bfseries}]
\node[main node] (1) {I};
\node[main node] (2) [right of=1] {J};
\node[main node] (3) [below left of=2] {K};
\path[every node/.style={font=\sffamily\small}]
    (1) edge node [left] {} (2)
    (2) edge node [left] {} (3)
    (3) edge node [left] {} (1);
\end{tikzpicture}
\end{figure}
\end{proof}

The \textit{intersection graph} of $R$, denoted
by $\Gamma(R)$, is a graph whose vertices are non-zero proper ideals of $R$ and two distinct
vertices $I$ and $J$ are adjacent if and only if $I \cap J\not=0$ \cite{ITS09}.
\begin{rem}\label{2.661}
Since the second ideals of $R$ are non-zero, $SII(R)$ is a subgraph of $\Gamma(R)$.
This subgraph is not necessarily  induced subgraph. For example, as we can see in the Figure B in the Example \ref{000}, $\langle q \rangle$ is not adjacent to $\langle pq \rangle$ in the graph $SII(\Bbb Z_{p^3q})$. But $\langle q \rangle$ is adgacent to $\langle pq \rangle$ in the graph $\Gamma(\Bbb Z_{p^3q})$.
\end{rem}

An \textit{Eulerian graph} is a graph which has a path that visits each edge
exactly once which starts and ends on the same vertex. A connected
non-empty graph is Eulerian if and only if the valency of each vertex is even \cite[Theorem 4.1]{BM76}.
\begin{rem}\label{2.6861}
In \cite[Theorem 5.1]{ITS09}, it is shown that $\Gamma(\Bbb Z_n)$ is Eulerian if and only if $n=p_1p_2 \ldots p_m$ or $n=p_1^{n_1}p_2^{n_2}\ldots p_m^{n_m}$, where each $n_i$ is even ($n_i \in \Bbb N$, $i=1,2,\ldots m$) and $p_i$'s are distinct primes.
But as we can see in Figure D, $SII(\Bbb Z_n)$ is not Eulerian when $n=p^2q^2$. One can see that if 
 $n=p_1p_2\ldots p_m$, where $p_i$'s are distinct primes, then $SII(\Bbb Z_n)$ is an Eulerian graph.
\end{rem}

 \begin{thm}\label{2.2}
Let $R$ be a commutative ring in which every ideal contains a minimal ideal (e.g., when $R$ is an Artinian ring). Then
$SII(R)$ has a universal vertex if and only if one
of the two statements hold:
\begin{itemize}
\item [(a)] $R$ has exactly one minimal ideal.
\item [(b)] $R$ has exactly two minimal ideals $M_1$ and $M_2$ such that $M_1 + M_2$ is a non-trivial
maximal ideal and that there is no non-second ideal that properly contained in $M_1+M_2$.
\end{itemize}
\end{thm}
\begin{proof}
Let (a) hold and $M$ be the minimal ideal of $R$. Then for each ideal $I$ of $R$, we have $I \cap M = M$, which is a second ideal and hence $I$ is adjacent to $M$. Thus, $M$ is a universal vertex.

Let (b) hold and set $I:= M_1 + M_2$. Assume that $J$ is a non-trivial ideal other than $I$
and without loss of generality, let $M_1\subseteq J$. Since by assumption, $I$ is a maximal ideal, we have
$M_1\subseteq I \cap J\subset I$. Now, by assumption, $I\cap J$ is a second ideal of $R$ and so $I$ is a universal vertex.

Conversely, let $SII(R)$ have a universal vertex, say $I$. If $R$ has a unique minimal
ideal, the proof is done. Now, assume that $R$ has at least three minimal ideals, say
$M_1,M_2$ and $M_3$. Note that $I$ cannot be a minimal ideal as two distinct minimal
ideals are not adjacent. Since $I$ is not a minimal ideal, then it is contains a minimal ideal, say $M_1$. If possible, let $M_2 \not \subseteq I$. Then by minimality of $M_2$, we have $I \cap M_2= 0$. Thus $I$ is not adjacent to $M_2$, a contradiction. Hence, $M_1+M_2+M_3\subseteq I$. Now, one of the following two cases holds.

\textbf{Case 1.} Let $I \not= M_1 +M_2+M_3$. Since $I\cap (M_1 +M_2+M_3)=M_1 +M_2+M_3$
and $I$ is a universal vertex in $SII(R)$, we get that $M_1 +M_2+M_3$ is a second ideal of $R$.
Thus $M_1 +M_2+M_3\subseteq (0:_MAnn_R(M_1)Ann_R(M_2)Ann_R(M_3))$ implies that $M_1=M_2=M_3$, which is a contradiction.

\textbf{Case 2.} Let $I = M_1 +M_2+M_3$.  Set $T = M_1+M_2$. Since sum of two minimal
ideals cannot be a second ideal, $I \cap T = T$ is not a second ideal, i.e.  $I$ is not adjacent to $T$, which contradicts
with our assumption that $I$ is a universal vertex. Therefore, $R$ has exactly two
minimal ideals, say $M_1$ and $M_2$. By the same argument as above, we conclude that
$I = M_1 +M_2$. Now, we show that $I$ is a maximal ideal. If possible, let there exists a non-trivial
ideal $J$ of $R$ such that $M_1 +M_2=I\subset J$. Then, as $I$ is a universal vertex, $J \cap I = I = M_1 +M_2$ is a
second ideal of $R$, which is a contradiction. Thus, $I$ is a maximal ideal.
If possible, let $J$ be a non-second ideal such that $M_i \subset J \subset I = M_1 +M_2$, where
$i = 1$ or $2$. But, it follows that $I \cap J = J$, a non-second ideal and hence $I$ is not adjacent to $J$, a
contradiction. Thus, there does not exist such ideal $J$ and the proof is completed.
\end{proof}

Recall that $R$ is said to be \textit{coreduced ring} if $rR=r^2R$ for each $r\in  R$ \cite {MR3755273}.

For an ideal $I$ of $R$ the \emph{second radical} (or \emph{second socle}) of $I$ is defined as the sum of all second ideals of $R$ contained in $I$ and it is denoted by $sec(I)$. In case $I$ does not contain any second ideal, the second radical of $I$ is defined to be $(0)$ (see \cite{CAS13} and \cite{AF11}). If $sec(R)=R$, then $R$ is coreduced \cite[Proposition 2.22]{MR3755273}.
\begin{rem}\label{2.3}
Let $R$ be a commutative ring in which every ideal contains a minimal ideal and $R$ is not a coreduced ring.
Then $sec(R)$ is adjacent to each
element of $Spec^s(R)$, where $Spec^s(R)$ is the set of all second ideals of $R$.
\end{rem}
\begin{proof}
As $R$ is a commutative ring in which every ideal contains a minimal ideal and $R$ is not a coreduced ring, $sec(R)$ is a non-zero proper ideal of $R$. For each second ideal $I$ of $R$, we have $I\cap sec(R)=I$ is a second ideal of $R$. Thus $sec(R)$ is adjacent to each
element of $Spec^s(R)$.
\end{proof}

\begin{cor}\label{2.4}
Let $R$ be a commutative ring in which every ideal contains a minimal ideal and $R$ is not a coreduced ring.
Then $sec(R)$ is the only minimal ideal of $R$ if and only if $sec(R)$ is a universal vertex of $SII(R)$.
\end{cor}
\begin{proof}
This follows from Theorem \ref{2.2}.
\end{proof}

\begin{thm}\label{2.5}
Let $R$ be a commutative ring in which every ideal contains a minimal ideal.
An ideal $I$ of $R$ is an isolated vertex in $SII(R)$ if and only if $I$ is
a minimal as well as maximal ideal of $R$.
\end{thm}
\begin{proof}
Let $I$ be an isolated vertex in $SII(R)$. If $I$ is not a minimal ideal, then it is
properly contains a minimal ideal, say $M$ and $I \cap M = M$. Thus $I$ is adjacent to $M$ in $SII(R)$,
a contradiction, and so $I$ is a minimal ideal. If $I$ is not a maximal ideal, then there
exists an ideal $J$ of $R$ such that $I \subset J \subset R$ and $I \cap J = I$, which is a
second ideal. Hence $I$ is adjacent to $J$, a contradiction, and so $I$ is a maximal ideal of $R$.

Conversely, let $I$ be a minimal as well as maximal ideal of $R$. If possible, assume
that $I$ is not isolated in $SII(R)$. Then there exists a non-zero proper ideal $J$ of $R$ other than $I$ such that $I \cap J$ is a second ideal. If $I \not \subseteq J$, then as $I$ is minimal ideal of $R$, we
have $I\cap J =0$, which is not  a second ideal. On the other hand if $I \subseteq J$, then by maximality
of $I$, we have $J =R$ or $I = J$, a contradiction. Thus such an ideal $J$ does not
exist and hence $I$ is an isolated vertex in $SII(R)$.
\end{proof}

\begin{thm}\label{2.6}
Let $R$ be a commutative ring in which every ideal contains a minimal ideal. Then
$SII(R)$ is complete if and only if $R$ has exactly one minimal ideal and every non-zero
non-second ideal is a maximal ideal.
\end{thm}
\begin{proof}
Let $SII(R)$ be a complete graph. Then by Theorem \ref{2.2}, $SII(R)$ has a universal vertex and
hence either one of the two conditions holds. Since two distinct minimal
ideals cannot be adjacent, $R$ cannot have two minimal ideals. Hence part (a) of
Theorem \ref{2.2} holds. Let $I$ be a non-zero proper ideal of $R$ which is not second. If
possible, there exists an ideal $J$ of $R$ such that $I \subset J \subset R$, then $I \cap J = I$. Since
$I$ is not a second ideal, $I$ is not adjacent to $J$, a contradiction to the completeness of $SII(R)$. Thus
every non-zero non-second ideal is a maximal ideal.

Conversely, let $R$ has exactly one minimal ideal, say $M$, and every
non-zero non-second ideal be a maximal ideal. Let $I_1$ and $I_2$ be two distinct non-zero
proper ideals of $R$. Then $M \subseteq I_1$ and $M\subseteq I_2$. Suppose $M\subseteq I_1 \cap I_2 = I_3$ is not a second ideal of $R$. Then by given condition, $I_3$ is a maximal ideal. But $I_3 \subseteq I_1, I_2$ and $I_1$, $I_2$ are proper
ideals of $R$.  Therefore, $I_1 = I_2 = I_3$, which is a contradiction. Thus $I_1 \cap I_2$ is second and so $I_1$ is adjacent to $I_2$. Hence, $SII(R)$ is complete.
\end{proof}

\begin{ex}\label{2.7996}
Let $k\geq 4$ be a positive integer.
Consider the ring $\Bbb Z_{p^k}$, where $p$ is prime. Since $p^2\Bbb Z_{p^k}$ is a non-second and non-maximal ideal of $\Bbb Z_{p^k}$, we have $SII(\Bbb Z_{p^k})$ is not a complete graph by Theorem \ref{2.6}.
\end{ex}

\begin{cor}\label{2.76}
Let $n$ be a positive integer. Then
 $SII(\Bbb Z_n)$ is a complete graph if and only if $n=p^k$, where $p$ is prime and $k=2$ or $k=3$.
\end{cor}
\begin{proof}
This follows from \cite[Theorem 2.9]{ITS09} and Example \ref{2.7996}.
\end{proof}

\begin{thm}\label{2.799}
Let $n$ be a positive integer. Then
the graph $SII(\Bbb Z_n)$ is disconnected if and only if $n=pq$, where $p$ and $q$ are distinct primes.
\end{thm}
\begin{proof}
Clearly, the graph $SII(\Bbb Z_{pq})$, where $p$ and $q$ are distinct primes is disconnected.
Now, let $n=p_1p_2...p_k$, where $p_i$'s are primes but may not be all distinct.
First assume that $k\geq 3$.
Set $m=p_2...p_k$. One can see that $\langle m\rangle$ is a second ideal of $R$ and $\langle m\rangle$ is adjacent to
$\langle p_2\rangle$. Therefore, for $k\geq 3$, the graph $SII(\Bbb Z_n)$ is connected.
If $k=2$ and $p_1=p_2$, then $SII(\Bbb Z_{p_1^2})$ contains only a single point and hence
connected. Therefore $n$ is of the form $pq$, where $p$ and $q$ are distinct primes.
\end{proof}

\begin{thm}\label{2.7}
Let $R$ be a commutative ring in which every ideal contains a minimal ideal. Then $SII(R)$ is connected if and only if $R$ is not a
direct sum of two of its minimal ideals. If $SII(R)$ is connected, then diam$(SII(R)) \leq 2$.
\end{thm}
\begin{proof}
Let $I_1$, $I_2$ be two proper ideals of $R$. If $I_1\cap I_2$ is a second ideal, then $I_1- I_2$
in $SII(R)$. Otherwise, assume that $I_1\cap I_2$ is not a second ideal. By assumption, both ideals $I_1$ and $I_2$ are contain a minimal ideals of $R$. If they are contain the same minimal ideal, say $M$, then we have $I_1 - M - I_2$, as minimal ideals
are second and hence $d(I_1, I_2) = 2$. Thus, we assume that they are not contain the same minimal ideals, $M_1\subset I_1$, $M_2\subset I_2$ and $M_1 \not= M_2$ are two minimal ideals of $R$.
First assume that $M_1+ M_2 \not=R$. By the minimality of $M_2$, we have $M_2 \cap I_1=0$.
Then $I_1 \cap (M_1+M_2)=I_1\cap M_1+I_1\cap M_2=M_1$ and so $I_1 -  M_1 + M_2$. Similarly,
 $M_1 + M_2-I_2$. Thus we have a path of length 2 given
by $I_1 -  M_1 + M_2 - I_2$ and so  $d(I_1, I_2) \leq 2$. Hence,  diam$(SII(R)) \leq 2$. If $M_1+ M_2 =R$, then since $M_1$ and $M_2$ are minimal ideals, the sum is directed.

Conversely, if $R$ is direct sum of two minimal ideals $M_1$ and $M_2$, the only non-trivial ideals
of $R$ are ${0}+M_2$ and $M_1+{0}$. In this case, $SII(R)$ consists of two isolated vertices
and hence it is not connected.
\end{proof}

Recall that $R$ is said to be a \textit{comultiplication ring} if for each ideal $I$ of $R$, we have $I=Ann_R(Ann_R(I))$ \cite{MR3934877}.
\begin{cor}\label{2.8}
Let $R$ be a comultiplication ring in which $R$ is not a
direct sum of two of its minimal ideals. Then diam$(SII(R)) \leq 2$.
\end{cor}
\begin{proof}
Every ideal in a comultiplication ring contains a minimal ideal by \cite[Theorem 7]{MR3934877}. Now the result follows from Theorem \ref{2.7}.
\end{proof}

In \cite [Theorem 2.6]{saha2023prime}, it is shown that the graph $PIS(R)$ is connected if and only if $R$ is not a
direct sum of two fields. If $PIS(R)$ is connected, then diam($PIS(R)) \leq 4$. Moreover,
if $R$ is a principal ideal ring, then diam($PIS(R)) \leq 2$. But as we see in the following theorem (the most of its proof is the proof of \cite[Theorem 2.6]{saha2023prime}) it is not necessary for $R$ to be a principal ideal ring.
\begin{thm}\label{2.5}
The graph $PIS(R)$ is connected if and only if $R$ is not a
direct sum of two fields. If $PIS(R)$ is connected, then diam($PIS(R)) \leq 2$. 
\end{thm}
\begin{proof}
Let $I_1$, $I_2$ be two non-zero ideals of $R$. If $I_1+ I_2$ is a prime ideal, then $I_1- I_2$
in $PIS(R)$. Otherwise, assume that $I_1+ I_2$ is not a prime ideal. Both ideals $I_1$ and $I_2$ are contained in some maximal ideals of $R$. If they are contained in the same maximal ideal, say $M$, then we have $I_1 - M - I_2$, as maximal ideals
are prime ideals and hence $d(I_1, I_2) = 2$. Thus, we assume that they are not contained in the same maximal ideals, $I_1\subset M_1$, $I_2\subset M_2$ and $M_1 \not= M_2$ are two maximal ideals of $R$.
First assume that $M_1\cap M_2 \not=0$. By the maximality of $M_2$, we have $M_2 + I_1=R$.
Then 
$$
M_1=M_1 \cap R=M_1 \cap (M_2+I_1)=(I_1\cap M_1)+(M_1\cap M_2)=I_1+(M_1 \cap M_2)
$$
and so $I_1 -  M_1 \cap M_2$. Similarly,
 $M_1 \cap M_2-I_2$. Thus we have a path of length 2 given
by $I_1 -  M_1 \cap M_2 - I_2$ and so  $d(I_1, I_2) \leq 2$. Hence, diam($PIS(R)) \leq 2$. If $M_1\cap M_2 =0$, then since $M_1$ and $M_2$ are maximal ideals, $R$ is a direct sum of two fields.

Conversely, if $R$ is direct sum of two two fields $F_1$ and $F_2$, the only non-trivial ideals
of $R$ are ${0}+F_2$ and $F_1+{0}$. In this case, $PIS(R)$ consists of two isolated vertices
and hence it is not connected.
\end{proof}

\begin{thm}\label{2.9}
If any two non-comparable ideals are adjacent in $SII(R)$, then
girth($SII(R)$) = 3.
\end{thm}
\begin{proof}
Let $I_1$ and $I_2$ be two non-comparable ideals which are adjacent in $SII(R)$.
Then $I_1\cap I_2$ is a second ideal of $R$. Since $I_1$ and $I_2$ are non-comparable, then $I_1$, $I_2$, and $I_1\cap I_2$ forms a triangle in $SII(R)$, i.e. girth($SII(R)$) = 3.
\end{proof}

\begin{cor}\label{2.10}
If $SII(R)$ is acyclic or girth$(SII(R)) > 3$, then no two non-comparable
ideals of $R$ are adjacent in $SII(R)$ and adjacency occurs only in case
of comparable ideals, i.e. for any edge in $SII(R)$, one of the terminal vertices is a
second ideal of $R$.
\end{cor}

\begin{thm}\label{2.11}
If girth($SII(R)$) = n, then there exist at least
$\lfloor n/2 \rfloor$ distinct second ideals in $R$.
\end{thm}
\begin{proof}
By Theorem \ref{2.9}, if two non-comparable ideals are adjacent in $SII(R)$,
then girth($SII(R)$) = 3 and the intersection of those two non-comparable ideals forms a
second ideal, and hence $R$ contains at least
$\lfloor 3/2 \rfloor=1$ second ideal. Thus, we assume
that girth$(SII(R)) > 3$, i.e. by Corollary \ref{2.10}, adjacency occurs only in case of
comparable ideals. Let $I_1 - I_2 - I_3- \cdots - I_n - I_1$ be a cycle of length $n$. First, we
observe that neither $I_1 \subseteq I_2 \subseteq  I_3\subseteq \cdots \subseteq I_n \subseteq I_1$ nor $I_1 \supseteq I_2 \supseteq I_3\supseteq \cdots\supseteq I_n \supseteq I_1$ can hold, as in both the cases all the ideals will be equal. Thus, without loss of
generality, we have $I_2 \subseteq  I_1$ , $I_2\subseteq  I_3$ and $I_2$ is a second ideal of $R$. Hence, we have the following
two cases:

\textbf{Case I.} Let $I_2 \subseteq  I_1, I_3$ and $I_4 \subseteq I_3$. Then, we have $I_2$, $I_4$ to be second ideals.

\textbf{Case II.} Let $I_2 \subseteq  I_1, I_3$ and $I_3 \subseteq I_4$. Then, we have $I_2$, $I_3$ to be second ideals.
In any case, we get at least 2 ideals to be second in $R$ among $I_1$, $I_2$, $I_3$ and $I_4$.
Continuing in this manner till $I_n$, we get at least
$\lfloor n/2 \rfloor$ ideals which are second in $R$.
\end{proof}

\begin{cor}\label{2.12}
Let $R$ has $k$ second ideals. Then $SII(R)$ is either acyclic or
girth$(SII(R)) \leq 2k$.
\end{cor}

Let $G$ be a graph. A non-empty subset $D$ of the
vertex set $V(G)$ is called a \textit{dominating set} if every vertex $V (G\setminus D)$ is adjacent to at least
one vertex of $D$. The \textit{domination number} $\gamma(G)$ is the minimum cardinality among the dominating sets of $G$.
\begin{thm}\label{2.13}
Let $R$ be a commutative ring in which every ideal contains a minimal ideal and let $\mathcal{M}$ be the set of all minimal ideals of $R$. Then $\mathcal{M}$ is a minimal dominating set of $SII(R)$ and $\gamma(SII(R)) \leq |\mathcal{M}|$. Moreover, $\gamma(SII(R)) = 1$ if and only if $R$ has exactly one minimal ideal or $R$ has exactly two minimal ideals $M_1$ and $M_2$ such
that $M_1 + M_2$ is a non-trivial maximal ideal such that there is no non-second ideal
properly contained $M_1 + M_2$. Also, if $R$ has exactly two minimal ideals which does
not satisfy the above condition, then $\gamma(SII(R)) = 2$.
\end{thm}
\begin{proof}
Since any ideal $I$ of $R$ is contains some element $M$ of $\mathcal{M}$ and $I\cap M = M$,
which is a second ideal, $\mathcal{M}$ dominates $SII(R)$. Let $M \in \mathcal{M}$. It is to be observed that
$\mathcal{M}\setminus \{M\}$ does not dominate $M$ and so fails to dominate $SII(R)$. Thus $\mathcal{M}$ is a
minimal dominating set of $SII(R)$ and $\gamma(SII(R)) \leq |\mathcal{M}|$. The second  and third parts follow
from Theorem \ref{2.2}.
\end{proof}

\begin{ex}\label{001}
Let $R=\Bbb Z_{pqr}$, where $p$, $q$, and $r$ are primes. Then the non-zero proper ideals of $R$ are $\langle p \rangle$, $\langle q \rangle$, $\langle r \rangle$, $\langle pq  \rangle$, $\langle pr \rangle$, and $\langle qr \rangle$. In the following figures,  we can see the graphs $PIS(\Bbb Z_{pqr}))$ and $SII(\Bbb Z_{pqr})$.
\begin{figure}[H]
\centering
\begin{subfigure}[b]{0.49\textwidth}
\centering
\caption{$PIS(\Bbb Z_{pqr}))$.}
\begin{center}
\begin{tikzpicture}
[auto,node distance=2 cm,
  thick,main node/.style={circle,fill=black!10,font=\sffamily\tiny\bfseries}]
\node[main node] (1) {$\langle r \rangle$};
\node[main node] (2) [below left of=1] {$\langle qr \rangle$};
\node[main node] (3) [below right of=1] {$\langle pr \rangle$};
\node[main node] (4) [below left of=2] {$\langle q \rangle$};
\node[main node] (5) [below right of=2] {$\langle pq \rangle$};
\node[main node] (6) [below right of=3] {$\langle p \rangle$};
\path[every node/.style={font=\sffamily\small}]
    (1) edge node [left] {} (2)
    (1) edge node [left] {} (3)
    (2) edge node [left] {} (4)
    (2) edge node [left] {} (5)
    (2) edge node [left] {} (3)
              (3) edge node [left] {} (5)
              (3) edge node [left] {} (6)
              (4) edge node [left] {} (5)
              (5) edge node [left] {} (6);

\end{tikzpicture}
\end{center}
\end{subfigure}
\hfill
\begin{subfigure}[b]{0.49\textwidth}
\centering
\caption{$SII(\Bbb Z_{pqr})$.}
\begin{center}
\begin{tikzpicture}
[auto,node distance=2 cm,
  thick,main node/.style={circle,fill=black!10,font=\sffamily\tiny\bfseries}]
\node[main node] (1) {$\langle pq \rangle$};
\node[main node] (2) [below left of=1] {$\langle p \rangle$};
\node[main node] (3) [below right of=1] {$\langle q \rangle$};
\node[main node] (4) [below left of=2] {$\langle pr \rangle$};
\node[main node] (5) [below right of=2] {$\langle r \rangle$};
\node[main node] (6) [below right of=3] {$\langle qr \rangle$};
\path[every node/.style={font=\sffamily\small}]
    (1) edge node [left] {} (2)
    (1) edge node [left] {} (3)
    (2) edge node [left] {} (4)
    (2) edge node [left] {} (5)
    (2) edge node [left] {} (3)
              (3) edge node [left] {} (5)
              (3) edge node [left] {} (6)
              (4) edge node [left] {} (5)
              (5) edge node [left] {} (6);
\end{tikzpicture}
\end{center}
\end{subfigure}
\end{figure}
 \end{ex}

 \begin{rem}\label{2.14}
The inequality in the Theorem \ref{2.13} may be strict. For example, in the graph $SII(\Bbb Z_{pqr})$, we have
$\{\langle p \rangle, \langle qr \rangle\}$ or $\{\langle r \rangle, \langle pq \rangle\}$ or $\{\langle q \rangle, \langle pr \rangle\}$ forms a dominating set  of $SII(\Bbb Z_{pqr})$. But $\Bbb Z_{pqr}$ have three minimal ideals $\langle pq \rangle$, $\langle pr  \rangle$, and  $\langle qr \rangle$.
\end{rem}

\begin{prop}\label{2.15}
If $R$ and $S$ are two isomorphic commutative rings with unity, then $SII(R)$ and $SII(S)$ are isomorphic as graphs.
\end{prop}
\begin{proof}
This is straightforward.
\end{proof}

\begin{ex}\label{002}
Let $R=\Bbb Z_{12}$ and $S=\Bbb Z_{18}$. Then the non-zero proper ideals of $R$ are $\langle 2 \rangle$, $\langle 3  \rangle$, $\langle 4 \rangle$, and $\langle 6  \rangle$. Also, the non-zero proper ideals of $S$ are
$\langle 2 \rangle$, $\langle 3  \rangle$, $\langle 6 \rangle$, and $\langle 9  \rangle$.
Clearly, the rings $R$ and $S$ are not isomorphic as rings. However, both of their corresponding second ideal intersection graphs are isomorphic as we can see in the following figures.
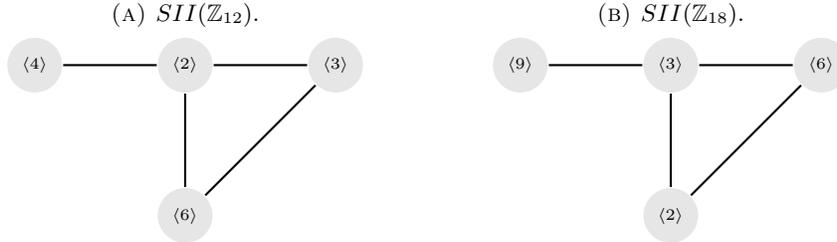
\begin{figure}[H]
\centering
\begin{subfigure}[b]{0.49\textwidth}
\centering
\caption{$SII(\Bbb Z_{12})$.}
\begin{center}
\begin{tikzpicture}
[auto,node distance=2 cm,
  thick,main node/.style={circle,fill=black!10,font=\sffamily\tiny\bfseries}]
\node[main node] (1) {$\langle 4 \rangle$};
\node[main node] (2) [right of=1] {$\langle 2 \rangle$};
\node[main node] (3) [right of=2] {$\langle 3 \rangle$};
\node[main node] (4) [below of=2] {$\langle 6 \rangle$};
\path[every node/.style={font=\sffamily\small}]
    (1) edge node [left] {} (2)
    (2) edge node [left] {} (3)
    (2) edge node [left] {} (4)
    (3) edge node [left] {} (4);
               \end{tikzpicture}
\end{center}
\end{subfigure}
\hfill
\begin{subfigure}[b]{0.49\textwidth}
\centering
\caption{$SII(\Bbb Z_{18})$.}
\begin{center}
\begin{tikzpicture}[auto,node distance=2 cm,
  thick,main node/.style={circle,fill=black!10,font=\sffamily\tiny\bfseries}]
\node[main node] (1) {$\langle 9 \rangle$};
\node[main node] (2) [right of=1] {$\langle 3 \rangle$};
\node[main node] (3) [right of=2] {$\langle 6 \rangle$};
\node[main node] (4) [below of=2] {$\langle 2 \rangle$};
\path[every node/.style={font=\sffamily\small}]
    (1) edge node [left] {} (2)
    (2) edge node [left] {} (3)
    (2) edge node [left] {} (4)
    (3) edge node [left] {} (4);
\end{tikzpicture}
\end{center}
\end{subfigure}
\end{figure}
 \end{ex}

\begin{prop}\label{2.16}
Let $R$ be a comultiplication ring and $I$ , $J$ be ideals of $R$. Then we have the following.
\begin{itemize}
\item [(a)] If $Ann_R(I \cap J)=I+J$, then
$I$ and $J$ are adjacent in $SII(R)$ if and only if they are adjacent in $PIS(R)$.
\item [(b)] If $Ann_R(I)=J$ and  $Ann_R(J)=I$, then
$I$ and $J$ are adjacent in $SII(R)$ if and only if they are adjacent in $PIS(R)$.
\item [(c)] $I$ and $Ann_R(I)$ are adjacent in $SII(R)$ if and only if they are adjacent in $PIS(R)$.
\end{itemize}
\end{prop}
\begin{proof}
(a) This follows from the fact that in the comultiplication ring $R$ an ideal $I$ is second if and only if $Ann_R(I)$ is a prime ideal of $R$ \cite[Theorem 196 (a)]{MR3934877}.

(b) We have $I+J=Ann_R(J)+Ann_R(I)=Ann_R(I \cap J)$ by \cite[Proposition 12 (b)]{MR3934877}. Now the result follows from part (a)

(c) Since $R$ is a comultiplication ring, $I=Ann_R(Ann_R(I))$. Now the result follows from part (b).
\end{proof}

\begin{thm}\label{2.17}
Let $R$ be a comultiplication ring. Then we have $PIS(R)\cong SII(R)$.
\end{thm}
\begin{proof}
Let be $\mathfrak{A}$ the set of all non-zero proper ideals of $R$. Define the map $\phi:V(PIS(R))=\mathfrak{A}\rightarrow V(SII(R))=\mathfrak{A}$ by $\phi (I)=Ann_R(I)$ for each $I \in \mathfrak{A}$. Then as $R$ is a comultiplication ring, one can see that $R$ is an isomorphism. Now let $I$ and $J$ be two non-zero proper ideal of $R$ such that
 $I$ is adjacent to $J$ in $PIS(R)$. Then $I+J$ is a prime ideal of $R$. But by \cite[Proposition 12 (a)]{MR3934877}, we have
$$
I+J=Ann_R(Ann_R(I))+Ann_R(Ann_R(J))=Ann_R(Ann_R(I)\cap Ann_R(J)).
$$
This implies that $Ann_R(I)\cap Ann_R(J)$ is a second ideal of $R$ \cite[Theorem 196 (a)]{MR3934877}. Therefore
$Ann_R(I)$ and $Ann_R(J)$ are adjacent in $SII(R)$, as needed.
\end{proof}

\begin{cor}\label{2.18}
Let $n$ be a positive integer. Then for the ring $\Bbb Z_n$ we have $PIS(\Bbb Z_n)\cong SII(\Bbb Z_n)$.
\end{cor}
\begin{proof}
By \cite[Example 11 (b)]{MR3934877}, $\Bbb Z_n$ is a comultiplication ring.
Thus the result follows from Theorem \ref{2.17}
\end{proof}

The following example shows that the condition $R$ is a comultiplication ring is required in the Theorem \ref{2.17}.
\begin{ex}\label{2.19}
Consider $R=\Bbb Z$. Then $\Bbb Z$ is not a comultiplication ring, 
$E(SII(\Bbb Z))=\emptyset$, and the ideal $2\Bbb Z$ is adjacent to the ideals $2k\Bbb Z$  in $PIS(\Bbb Z)$ for each positive integer $k>1$. 
\end{ex}
\textbf{Author Contribution.}
All the authors have an equal contribution. 

\textbf{Ethics approval.}
This article does not contain any studies with human participants or animals performed by the  author.

\textbf{Conflict Of Interest.}
The author declares that there are no conflicts of interest.

\textbf{Data Availability.}
Data sharing not applicable to this article as no datasets were generated or analysed during the current study.

\textbf{Funding.}
The author has not disclosed any funding.

\bibliographystyle{amsplain}

\end{document}